\documentclass[12pt,a4paper]{article}

\usepackage{graphicx}
\usepackage{amsmath}
\usepackage{amssymb}
\usepackage{hyperref}
\usepackage[all]{hypcap}
\usepackage{doi}

\usepackage{amsthm}
\newtheorem{theorem}{Theorem}
\newtheorem{proposition}[theorem]{Proposition}
\newtheorem{lemma}[theorem]{Lemma}
\theoremstyle{definition}
\newtheorem{definition}[theorem]{Definition}

\let\vec\boldsymbol
\newcommand{\Ord}[1]{\ensuremath{\mathcal O
  \mathchoice{\big(#1\big)}{\big(#1\big)}{(#1)}{(#1)}
  }}
\newcommand{\ord}[1]{\ensuremath{o
  \mathchoice{\big(#1\big)}{\big(#1\big)}{(#1)}{(#1)}
  }}
\newcommand{\ode}{\textsc{ode}}
\newcommand{\sde}{\textsc{sde}}
\newcommand{\RR}{\mathbb R}
\newcommand{\NN}{\mathbb N}
\newcommand{\cdi}{\textsc{cdi}}
\newcommand{\D}[2]{\mathchoice
  {\frac{\partial #2}{\partial #1}}
  {{\partial #2}/{\partial #1}}
  {{\partial #2}/{\partial #1}}
  {{\partial #2}/{\partial #1}}
  }
\newcommand{\opL}{\mathcal L}

\newcommand\z{\hspace{-0.6em}}

\title{Numerical integration of ordinary differential equations with rapidly oscillatory factors}

\author{
J.~E. Bunder\thanks{School of Mathematical Sciences, University of Adelaide, South Australia~5005, Australia. 
\protect\url{mailto:judith.bunder@adelaide.edu.au}}
\and 
A.~J. Roberts\thanks{School of Mathematical Sciences, University of Adelaide, South Australia~5005, Australia. 
\protect\url{mailto:anthony.roberts@adelaide.edu.au}}
}

\begin{document}
\maketitle

\begin{abstract}
We present a methodology for numerically integrating ordinary differential equations containing rapidly oscillatory terms.
This challenge is distinct from that for differential equations which have rapidly oscillatory solutions: here the differential equation itself has the oscillatory terms.
Our method generalises Filon quadrature for integrals, and is analogous to integral techniques designed to solve stochastic differential equations and, as such, is applicable to a wide variety of ordinary differential equations with rapidly oscillating factors. 
The proposed method flexibly achieves varying levels of accuracy depending upon the truncation of the expansion of certain integrals.
Users will choose the level of truncation to suit the parameter regime of interest in their numerical integration.\\
keywords: highly oscillatory problems, ordinary differential equations.
\end{abstract}

\section{Introduction}

Ordinary differential equations (\ode{}s) containing rapidly oscillatory terms are a challenge for numerical computation.
A separate much researched challenge are \ode{}s where the equations are not themselves rapidly oscillatory, but do have rapidly oscillatory solutions.
Here we focus on the case where the \ode{} contains both terms which rapidly oscillate on a microscale time and terms which vary smoothly over macroscale times of interest.
The microscale oscillating terms combined with the slow macroscale terms in the \ode{} produce solutions with multiscale structure.
Typically, solutions are smoothly varying over the macroscale, but with superimposed microscale detail (e.g., Figure~\ref{fig:egOsc}). 
Such microscale detail interacts via nonlinearity to modify the apparent macroscale behaviour.

We consider the class of \ode{}s for some function $u(t)\in\RR^m$ of the form 
\begin{equation}
\frac{du}{dt}=a(t,u)+b(t,u)v(t),\quad u(t_n)=u_{t_n}\,,
\label{eq:ode}
\end{equation}
for smoothly varying coefficient functions $a,b:\RR\times\RR^m\to\RR^m$, and where the `vacillating'~$v(t)$ is some given rapidly oscillating periodic scalar function of time~$t$ with zero mean, and constant oscillation frequency~$\varpi$ (that is, period~$2\pi\varpi^{-1}$). 
The rapidly oscillating~$v(t)$ may be functions such as~$\sin\varpi t$ or~$e^{i\varpi t}$.
Suppose we are interested in sampling the solution over a relatively long macroscale time, say over time steps of size~$h$.
We assume the microscale oscillation is rapid with respect to the macroscale time scale~$h$ so that $\varpi^{-1}\ll h$\,.
For definiteness, we also assume time~$t$ and unknown~$u$ have been scaled so that the coefficient functions $a$ and~$b$ vary on a scale of one in both $t$ and~$u$.
Figure~\ref{fig:egOsc} plots solutions from the example \ode~\eqref{eq:expDE}, discussed in Section~\ref{sec:examples}, that are in the class~\eqref{eq:ode} of the \ode{}s considered here.
In these examples the multiscale structure of the solution is clearly visible.
Over the macroscale time interval~$[0,1]$ the general trend of the solution is revealed, but over microscale time intervals of the order $\varpi^{-1}=0.01$ the solution is highly oscillatory.
Such differential equations arise in a wide variety of systems, including molecular dynamics~\cite{Hong2001}, circuit simulations~\cite{Condon2009b}, chemical reactions~\cite{Samoilov2005}, and weather systems~\cite{Penland2008}. 

\begin{figure}
\centering
\includegraphics[scale=1]{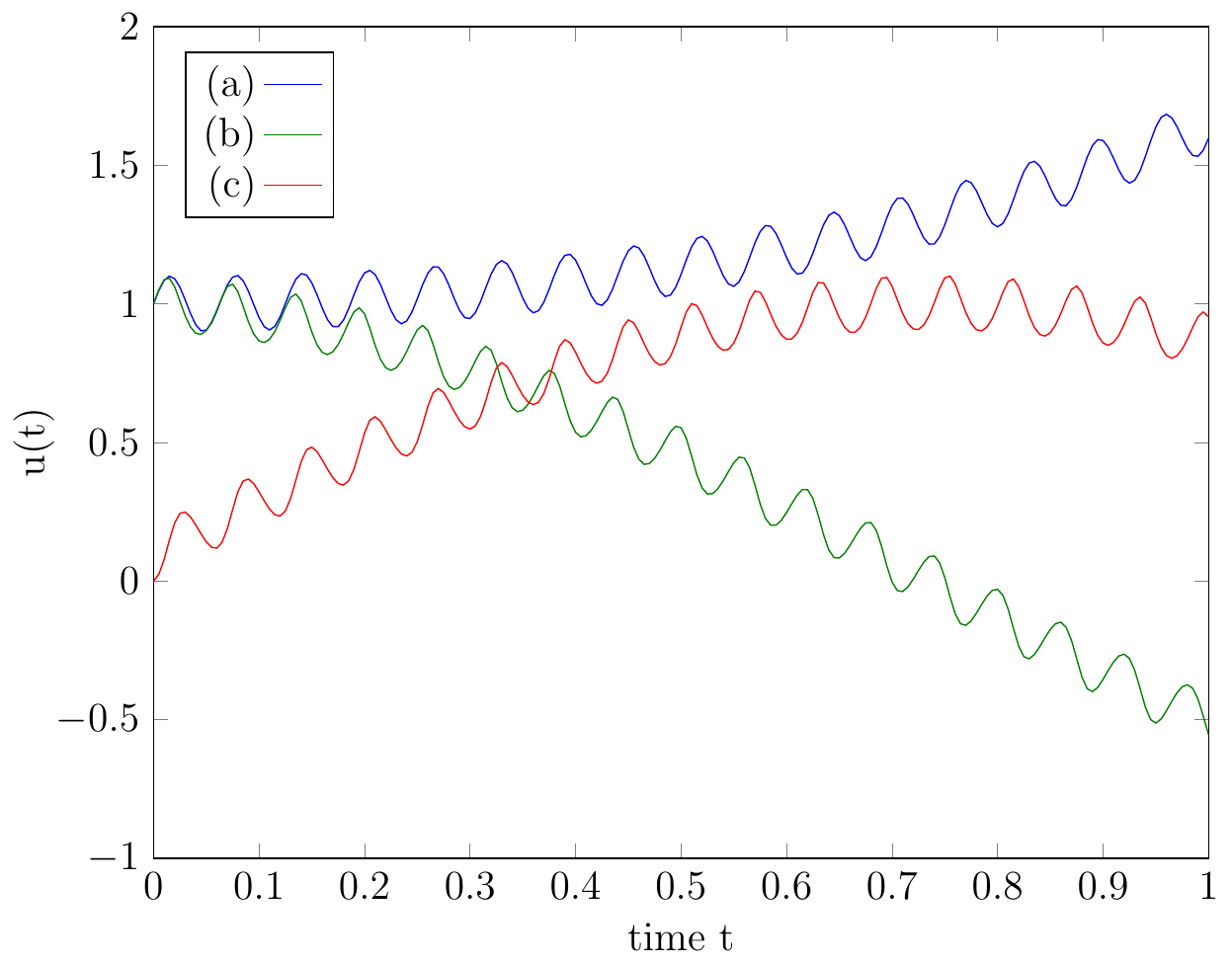}
\caption{Example solutions, using microscale time steps, of the \ode~\eqref{eq:expDE} for oscillations of strength $\mu=10$ and frequency $\varpi=100$, and initial condition $u(0)=1$: (a)~$\gamma=0$, $\alpha(t)=t$, $v(t)=\cos\varpi t$; (b)~real part for $\gamma=2$, $\alpha(t)=2i$, $v(t)=e^{i\varpi t}$; and (c)~imaginary part for $\gamma=2$, $\alpha(t)=2i$, $v(t)=e^{i\varpi t}$.}
\label{fig:egOsc}
\end{figure}

Established numerical techniques, such as Runge--Kutta or Gear's method, work well for \ode{}s without rapidly oscillating terms. 
But these techniques become computationally expensive when oscillations with microscale periods~$\varpi^{-1}$ are present, particularly when there is a significant difference between the two relevant time scales, $\varpi^{-1}\ll h$.
For example, Matlab's stiff \ode{} solver \verb|ode15s| takes $15$~time steps per microscale oscillation to reproduce the \ode{} solutions shown in Figure~\ref{fig:egOsc}.
For this reason, several numerical methods have recently been developed specifically for accurately and efficiently evaluating \ode{}s containing rapidly oscillating terms~\cite{Iserles2005, Iserles2006, Huybrechs2006, Olver2007} including a method by Condon, Dea\~no and Iserles~\cite{Condon2009a, Condon2010a, Condon2010b} which is discussed in Section~\ref{sec:exp-macro} as a comparison.
The aim herein is to develop efficient and flexible computational schemes where each time step spans many microscale oscillation periods.

Our novel method for numerically solving \ode{}s with rapidly oscillating terms is based upon iterating integrals. 
Section~\ref{sec:iterative} derives a multivariable Taylor series expansion for the solution at time~$t_{n+1}$, based about the solution at time~$t_n$ in powers of the typical microscale period of oscillation~$\varpi^{-1}$ and the macroscale time step~$h=(t_{n+1}-t_n)$ \cite{Kloeden2001}.
The integral approach empowers us to quantify the remainder term in the Taylor expansion, and hence empowers users to potentially bound the errors in any application of the approximation scheme.
In these systems, the macroscale time step~$h$ is  much longer than a typical microscale period of oscillation~$\varpi^{-1}$, so $h\varpi\gg 1$.
Validity of the Taylor expansion requires small enough $\varpi^{-1}$~and~$h$.
The method is analogous to a Taylor expansion scheme, originally developed for Ito stochastic differential equations (\sde{}s) governed by a Wiener process, which is achieved by an iterative application of the Ito formula~\cite{Kloeden2001, Kloeden1992}.

A typical Ito \sde{},
\begin{equation}
du=a(t,u)dt+b(t,u)dW_t\,,\label{eq:sde}
\end{equation}
where $W_t$~is a stochastic Wiener process, closely resembles the \ode~\eqref{eq:ode}. 
However, in contrast to the deterministic periodic function~$v(t)$ in \ode~\eqref{eq:ode}, the Wiener process oscillates over all time scales and is nondeterministic. 
Section~\ref{sec:wiener} relates our iterated integral method of \ode{}s~\eqref{eq:ode} to previously developed integral methods of \sde{}s~\eqref{eq:sde}.

For conciseness we adopt notation analogous to that used for \sde{}s.
Let subscripts~$t$ to refer to evaluation at time~$t$, and similarly for subscripts in other time-like quantities such as~$s$, $h$ and~$0$.
For example, $u_t$ is~$u$ at time~$t$ and, for some function~$f(t,u)$, $f_t$~denotes evaluation at $t$~and~$u_t$.
Define $dV_t:=v_t\,dt$ so that the integral of the oscillations
\begin{equation}
\int_{t_n}^{t_{n+1}}\z v_t\,dt=\int_{V_{t_n}}^{V_{t_{n+1}}}\z dV_t=V_{t_{n+1}}-V_{t_n}\,.
\end{equation} 
Without loss of generality, we assume that $v(t)$ and~$V(t)$ have zero mean so that $(V_{t_{n+1}}-V_{t_n})\sim\varpi^{-1}$~\footnote{This follows from equation~\eqref{eq:bound} with $f_s=1$.} (unless we also scale the strength of the oscillations~$v_t$ with the frequency~$\varpi$).
For cases where the mean $\langle v\rangle\neq 0$, for example $v(t)=e^{\cos\varpi t}$, we define $\tilde v(t):= v(t)-\langle v\rangle$ and $\tilde a(t,u):=a(t,u)+\langle v\rangle b(t,u)$ and use the tilde functions in the \ode~\eqref{eq:ode} so it retains the correct form but now has oscillations~$\tilde v(t)$ with zero mean.
The integral function~$V_t$ is analogous to the Wiener process in stochastic calculus (in such an analogy, the `vacillating'~$v_t$ would be analogous to the formal `white noise').

Herein we focus on the case when the \ode~\eqref{eq:ode} contains just one rapidly oscillating factor.
We expect the case when there are multiple rapidly oscillating factors to be similar, but more complicated to express, and leave the case for further research.
Such future research could consider some derivative-free schemes analogous to those which efficiently solve stochastic differential equations with a multidimensional Wiener process.
For example, stochastic Runge--Kutta methods for solving \sde{}s are not only efficient but also have good accuracy and stability~\cite[e.g.]{Roessler2010, Komori2012, Roberts2011b}.

\section{Iterative integration scheme}
\label{sec:iterative}

Consider the time derivative of any smooth function~$f(t,u):\RR\times\RR^m\to\RR^k$ for any~$k$, for which the variable~$u$ satisfies the rapidly oscillating \ode~\eqref{eq:ode}. 
We use the term ``smooth function'' to mean the class of functions differentiable as often as is needed for the expressions at hand (that is, restricted to a suitable Sobolev space).
Expand the time derivative of~$f(t,u)$ using the chain rule:%
\footnote{Equations~\eqref{eq:chainrule} and \eqref{eq:operators} invoke the standard inner product dot operator~``$\cdot$'': that is, $\D uf\cdot du/dt=\sum_{j=1}^m (\partial f/\partial u_j)(du_j/dt)$ and $a\cdot \partial/\partial u=\sum_{j=1}^ma_j\partial/\partial u_j$ for $a,u\in\RR^m$.
This dot product is implicit in all the operators~$\opL_t^0$ and~$\opL_t^1$.}
\begin{equation}
\frac{df}{dt}=\D tf+\D uf\cdot\frac{du}{dt}=\opL^0_tf(t,u)+v_t \opL^1_tf(t,u),\label{eq:chainrule}
\end{equation}
in terms of the two operators (analogous to those used for \sde{}s)
\begin{equation}
\opL^0_t=\left[\frac{\partial}{\partial t}+a\cdot \frac{\partial}{\partial u}\right]_t \quad\text{and}\quad
\opL^1_t=\left[b\cdot\frac{\partial}{\partial u}\right]_t.\label{eq:operators}
\end{equation}
The integral version of the chain rule~\eqref{eq:chainrule} for any smooth function~$f(t,u)$, integrated over the interval~$(t_n,t)$, is
\begin{equation}
f_{t}=f_{t_n}+\int_{t_n}^{t}\opL_s^0f_s\,ds+\int_{V_{t_n}}^{V_{t}}\opL_s^1f_s\,dV_s\,.\label{eq:CRint}
\end{equation}
We now show how successive iterations of the integral formula~\eqref{eq:CRint} lead to useful hierarchal integral expressions for the solution~$u(t)$ of the \ode~\eqref{eq:ode} at $t_{n+1}=t_n+h$\,.
The integral expression involves powers of the micro time scales~$h$ and~$\varpi^{-1}$.
Iteration of the formula~\eqref{eq:CRint} generates expressions with precise remainders for error estimation.

\paragraph{First integral approximation}

We start with the \ode~\eqref{eq:ode} integrated
over the temporal interval~$(t_n,t_{n+1})$,
\begin{align}
u_{t_{n+1}}=u_{t_n}+\int_{t_n}^{t_{n+1}}  \frac{du}{dt}\,dt
=u_{t_n}+\int_{t_n}^{t_{n+1}}\z a_t\,dt+\int_{t_n}^{t_{n+1}}\z b_t\,dV_t\,,
\end{align}
having substituted the \ode~\eqref{eq:ode} to obtain the last expression on the right-hand side.
Now invoke the formula~\eqref{eq:CRint} for the integrands of both integrals in the above equation, that is, for both $f_t=a_t$ and $f_t=b_t$\,. 
We obtain the first expansion
\begin{equation}
u_{t_{n+1}}=u_{t_n}+a_{t_n}\int_{t_n}^{t_{n+1}}\z dt+b_{t_n}\int_{V_{t_n}}^{V_{t_{n+1}}}\z dV_t
+R_{1,1}
\label{eq:first}
\end{equation}
where the remainder term is the sum of integrals
\begin{align}
R_{1,1}&:=\int_{t_n}^{t_{n+1}}\z \int_{t_n}^t\opL_s^0a_s\,ds\,dt
+\int_{t_n}^{t_{n+1}}\z \int_{V_{t_n}}^{V_t}\opL_s^1a_s\,dV_s\,dt
\nonumber\\&\quad{}
+\int_{V_{t_n}}^{V_{t_{n+1}}}\z \int_{t_n}^t\opL_s^0b_s\,ds\,dV_t
+\int_{V_{t_n}}^{V_{t_{n+1}}}\z \int_{V_{t_n}}^{V_t}\opL_s^1b_s\,dV_s\,dV_t\,.
\label{eq:firstr}
\end{align}
Upon evaluating the two integrals in formula~\eqref{eq:first} we obtain an estimate for~$u_{t_{n+1}}$, with second order errors in~$h$ and~$\varpi^{-1}$ since $(V_{t_{n+1}}-V_{t_n})=\Ord{\varpi^{-1}}$, namely
\begin{equation}
u_{t_{n+1}}=u_{t_n}+a_{t_n}h+b_{t_n}(V_{t_{n+1}}-V_{t_n})+\Ord{h^2+\varpi^{-2}}.
\label{eq:uh1st}
\end{equation}
The four remainder integrals in equation~\eqref{eq:firstr}  straightforwardly determines the order of local error in the time step~\eqref{eq:uh1st}.  
This remainder is negligible at first order in $h$~and~$\varpi^{-1}$ since each integral over a time variable provides an additional order of~$h$, and each integral over the oscillating function provides an additional order of~$\varpi^{-1}$; that is, the four neglected integrals are all of second order, or higher, in $h$~and/or~$\varpi^{-1}$, as demonstrated by Lemma~\ref{lem:1}.

\begin{definition}
Let $C^{k}(\Omega)$ denote the space of functions into~\(\RR^m\) on an open set~$\Omega\subset \RR^{m+1}$ which are continuous up to and including $k$th~order derivatives. 
Define the norm for~$C^{k}(\Omega)$, in terms of the vector \(p\)-norm $\|u\|_p=(|u_1|^p+\cdots +|u_{m}|^p)^{1/p}$ for $u\in \mathbb{R}^{m}$, as  
\begin{equation}
\|z\|_{C_p^{k}(\Omega)}:=\left(\sum_{|\alpha|\leq k}\|D^{\alpha}z\|^p_{p}\right)^{1/p}
\label{eq:norm}
\end{equation}
for all $z\in C^{k}(\Omega)$, for multi-index $\alpha=(\alpha_0,\ldots,\alpha_{m})\in\NN_0^{m+1}$, and where 
\begin{equation*}
D^{\alpha}z:=\frac{\partial^{|\alpha|}z}{\partial t^{\alpha_0}\partial u_1^{\alpha_1}\cdots\partial u_m^{\alpha_m}}
\quad\text{and}\quad |\alpha|:=\alpha_0+\cdots+\alpha_m\,.
\end{equation*}
\end{definition}

\begin{lemma}[first error bound]\label{lem:1}
Assume there exists an open domain~$\Omega\subset\RR^{m+1}$ such that $(t,u(t))\in \Omega$ over the time interval $t_n\leq t\leq t_{n+1}$\,. If $a,b\in C^{1}(\Omega)$ are bounded by $\|a\|_{C_p^{1}(\Omega)},\|b\|_{C_p^{1}(\Omega)}\leq K$\,,
then the error of the time step~\eqref{eq:uh1st} is bounded by
\begin{equation}
\|R_{1,1}\|_p\leq \tfrac12(K^2+K)h^2 
+(2K^2+K)\|v\|h\varpi^{-1}
+K^2\|v\|^2\varpi^{-2}
\label{eq:bdd11}
\end{equation}
where $\|v\|=2\pi\max_{0\leq t<2\pi\varpi^{-1}}|v_t|$.
\end{lemma}

\begin{proof}
Since $\|a\|_{C_p^{1}(\Omega)},\|b\|_{C_p^1(\Omega)}\leq K$\,, from the norm~\eqref{eq:norm} the $p$-norm of relevant derivative are bounded: $\|D^{\alpha}a\|_{p},\|D^{\alpha}b\|_{p}\leq K$ for multi-indices $|\alpha|\leq 1$. 
Now establish the bound that
\begin{equation}
\left\|\int_{V_{t_n}}^{V_{t}}f_s\,dV_s\right\|_{p}\leq \|f\|_{p}^{\infty}\|v\|\varpi^{-1},\label{eq:bound}
\end{equation}
where $\|f\|_{p,\infty}=\max_{t_n\leq s\leq t_{n+1}}\|f_s\|_p$ and $t_n\leq t\leq t_{n+1}$. 
This bound follows from
\begin{align*}
\left\|\int_{V_{t_n}}^{V_{t}}f_sdV_s\right\|_p&\leq \|f\|_{p,\infty}\left|\int_{t_n}^{t}v_s\,ds\right| \\
&= \|f\|_{p,\infty}\left|\int_{t_n}^{t_n+\tau}v_s\,ds\right|
\quad \text{for }0\leq\tau<2\pi\varpi^{-1}  \\
&\leq \|f\|_{p,\infty}\|v\|/2\pi\left|\int_{t_n}^{t_n+\tau}ds\right|\nonumber\\
&=\|f\|_{p,\infty}\|v\|\tau/2\pi \\
&\leq \|f\|_{p,\infty}\|v\|\varpi^{-1},
\end{align*}
where $p$-norms of integrals over $v_s$ are replaced with ordinary absolute values since $v_t,t\in \mathbb{R}$.

Since $\|\opL_s^0a_s\|_p,\|\opL_s^0b_s\|_p\leq (K+K^2)$ and $\|\opL_s^1a_s\|_p,\|\opL_s^1b_s\|_p\leq K^2$ within the time interval $t_n\leq t\leq t_{n+1}$, the remainder~\eqref{eq:firstr} is bounded by
\begin{align*}
\|R_{1,1}\|_p&\leq (K+K^2)\left(\left|\int_{t_n}^{t_{n+1}}\int_{t_n}^tds\,dt\right|+\left|\int_{V_{t_n}}^{V_{t_{n+1}}}\int_{t_n}^tds\,dV_t\right|\right) \\
&\quad{}+K^2\left(\left|\int_{t_n}^{t_{n+1}}\int_{V_{t_n}}^{V_t}dV_s\,dt\right|+\left|\int_{V_{t_n}}^{V_{t_{n+1}}}\int_{V_{t_n}}^{V_t}dV_s\,dV_t\right|\right).
\end{align*}
All the \(t\)~integrals are evaluated exactly and equation~\eqref{eq:bound} provides upper bounds for all $V_t$~integrals, with $\|t-t_n\|_{p,\infty}=h$ and $\|1\|_{p,\infty}=1$, to obtain the bound~\eqref{eq:bdd11} on the remainder~\eqref{eq:firstr}.
\end{proof}

\paragraph{Second integral approximation}

To estimate~$u_{t_n}$ to third order errors in $h$~and~$\varpi^{-1}$ we expand equation~\eqref{eq:first} further by applying formula~\eqref{eq:CRint} to the integrands~$\opL_s^0a_s$, $\opL_s^1a_a$, $\opL_s^0b_s$ and~$\opL_s^1b_s$ in the remainder~\eqref{eq:firstr}:
\begin{align}
u_{t_{n+1}}&=u_{t_n}+a_{t_n}\int_{t_n}^{t_{n+1}}\z dt+b_{t_n}\int_{V_{t_n}}^{V_{t_{n+1}}}\z dV_t+\opL_{t_n}^0a_{t_n}\int_{t_n}^{t_{n+1}}\z\int_{t_n}^tds\,dt
\nonumber\\&\quad
{}+\opL_{t_n}^1a_{t_n}\int_{t_n}^{t_{n+1}}\z \int_{V_{t_n}}^{V_t}dV_s\,dt+\opL_{t_n}^0b_{t_n}\int_{V_{t_n}}^{V_{t_{n+1}}}\z \int_{t_n}^tds\,dV_t
\nonumber\\&\quad
{}+\opL_{t_n}^1b_{t_n}\int_{V_{t_n}}^{V_{t_{n+1}}}\z \int_{V_{t_n}}^{V_t}dV_s\,dV_t
+R_{2,2}\,,
\label{eq:2nd}
\end{align}
where the new remainder is the sum of eight integrals, namely
\begin{align}
R_{2,2}&:=
\int_{t_n}^{t_{n+1}}\z \int_{t_n}^t\int_{t_n}^s\opL_r^0\opL_r^0a_r\,dr\,ds\,dt
+\int_{t_n}^{t_{n+1}}\z \int_{t_n}^t\int_{V_{t_n}}^{V_s}\opL_r^1\opL_r^0a_r\,dV_r\,ds\,dt
\nonumber\\&\quad
{}+\int_{t_n}^{t_{n+1}}\z \int_{V_{t_n}}^{V_t}\int_{t_n}^{s}\opL_r^0\opL_r^1a_r\,dr\,dV_s\,dt
+\int_{t_n}^{t_{n+1}}\z \int_{V_{t_n}}^{V_t}\int_{V_{t_n}}^{V_s}\opL_r^1\opL_r^1a_r\,dV_r\,dV_s\,dt
\nonumber\\&\quad
{}+\int_{V_{t_n}}^{V_{t_{n+1}}}\z \int_{t_n}^t\int_{t_n}^s\opL_r^0\opL_r^0b_r\,dr\,ds\,dV_t
+\int_{V_{t_n}}^{V_{t_{n+1}}}\z \int_{t_n}^t\int_{V_{t_n}}^{V_s}\opL_r^1\opL_r^0b_r\,dV_r\,ds\,dV_t
\nonumber\\&\quad
{}+\int_{V_{t_n}}^{V_{t_{n+1}}}\z \int_{V_{t_n}}^{V_t}\int_{t_n}^{s}\opL_r^0\opL_r^1b_r\,dr\,dV_s\,dV_t
+\int_{V_{t_n}}^{V_{t_{n+1}}}\z \int_{V_{t_n}}^{V_t}\int_{V_{t_n}}^{V_s}\opL_r^1\opL_r^1b_r\,dV_r\,dV_s\,dV_t\,.
\label{eq:2ndr}
\end{align}
We expect the six integrals in the time step~\eqref{eq:2nd} to be evaluated straightforwardly using the known properties of~$V_t$.
Then the eight integrals in the remainder~\eqref{eq:2ndr} provide the error when an estimate of~$u_{t_{n+1}}$ is required to third order errors in $h$~and~$\varpi^{-1}$, as demonstrated by Lemma~\ref{lem:2}.

\begin{lemma}[second error bound]\label{lem:2}
Assume there exists an open domain~$\Omega\subset\RR^{m+1}$ such that $(t,u(t))\in \Omega$ over the time interval $t_n\leq t\leq t_{n+1}$\,. 
If $a,b\in C^2(\Omega)$ are bounded by $\|a\|_{C_p^2(\Omega)},\|b\|_{C_p^2(\Omega)}\leq K$\,,
then the error of the time step~\eqref{eq:2nd} is bounded by
\begin{align}
\|R_{2,2}\|_p&\leq \tfrac16(2K^3+4K^2+K)h^3
+\tfrac12(8K^3+8K^2+K)\|v\|h^2\varpi^{-1} \nonumber\\
&\quad{}+(6K^3+4K^2)\|v\|^2h\varpi^{-2}
+2K^3\|v\|^3\varpi^{-3}.
\label{eq:bdd22}
\end{align}
\end{lemma}

\begin{proof}[Outline of proof]
Since $\|a\|_{C_p^2(\Omega)},\|b\|_{C_p^2(\Omega)}\leq K$\,, from equation~\eqref{eq:norm} the $p$-norms of derivatives are bounded: $\|D^{\alpha}a\|_{p},\|D^{\alpha}b\|_{p}\leq K$ for multi-indices $|\alpha|\leq 2$\,.
Substitute $\|\opL_r^0\opL_r^0f_r\|_p\leq (K+4K^2+2K^3)$, $\|\opL_r^1\opL_r^0f_r\|_p\leq(K^2+2K^3)$, $\|\opL_r^0\opL_r^1f_r\|_p\leq (2K^2+2K^3)$ and $\|\opL_r^1\opL_r^1f_r\|_p\leq 2K^3$ into the remainder~\eqref{eq:2ndr}, where function~$f_r$ is either \(a_r\) or~\(b_r\). 
Solve all temporal integrals exactly and use equation~\eqref{eq:bound} for upper bounds of all $V_t$~integrals, using $\|(t-t_n)^l\|_{p,\infty}=h^l$ for non-negative integer~$l\in\NN_0$. 
\end{proof}

\paragraph{Further integral approximations}

When higher orders of $h$~and~$\varpi^{-1}$ are required, one would continue expanding integrands using formula~\eqref{eq:CRint} until the desired order is reached.
When all terms containing $\kappa$ or fewer integrals are retained for the evaluation of $u_{t_{n+1}}$, then the remainder, denoted~$R_{\kappa,\kappa}$, contains all neglected integrals and so consists of terms containing $\kappa+1$~integrals. 
However, this expansion assumes we weight $h$ and~$\varpi^{-1}$ of equal importance in the expansion. 
In general, we truncate the expansions of the integrals in $h$~and~$\varpi^{-1}$ at different orders since $h$~and~$\varpi^{-1}$ need not be of a equal importance. 

Consider the regime where the microscale oscillation time $\varpi^{-1}\sim h^{\rho}$ for some real exponent $\rho>0$\,.
In this regime, suppose we wish to estimate~$u_{t_{n+1}}$ correct to~$\Ord{h^{\kappa}}$. 
Since  each integral over~$t$ adds order~$h$ and each integral over~$V_t$ adds order $\varpi^{-1}\sim h^{\rho}$, each retained term in the integral estimate of~$u_{t_{n+1}}$ must be composed of $q_0$~integrals over~$t$ and $q_1$ over~$V_t$ such that $q_0+q_1\rho\leq \kappa$. 
The error of such an estimate is the sum of the neglected integrals and is represented by the remainder~$R_{\kappa,\kappa/\rho}$. 
In general, we define the remainder~$R_{\kappa_0,\kappa_1}$ as the sum of the remaining integrals after the recursive integral expansion sufficient and necessary to estimate~$u_{t_{n+1}}$ so that all integrals with $q_0$~integrals over~$t$ and $q_1$ over~$V_t$ such that $q_0/\kappa_0+q_1/\kappa_1\leq 1$ have constant integrand (for general \(a\) and~\(b\)), as in equation~\eqref{eq:general}. 
The orders \(\kappa_0\) and~\(\kappa_1\) are chosen to suit the regime of application of the scheme.

\begin{proposition}[order of error]\label{lem:3}
Assume there exists an open domain~$\Omega\subset\RR^{m+1}$ such that $(t,u(t))\in \Omega$ over the time interval $t_n\leq t\leq t_{n+1}$\,. 
If $a,b\in C^{\max(\kappa_0,\kappa_1)}(\Omega)$ are bounded,
then the estimate~$u_{t_{n+1}}$ has error \(R_{\kappa_0,\kappa_1}=\ord{h^{\kappa_0}+\varpi^{-\kappa_1}}\).
\end{proposition}

\begin{proof}[Outline of proof]
Expand the integrals for~$u_{t_{n+1}}$ so that all integrals with $q_0$~integrals over~$t$ and $q_1$ over~$V_t$ such that $q_0/\kappa_0+q_1/\kappa_1\leq 1$ have constant integrand. 
Then the error, the remainder~$R_{\kappa_0,\kappa_1}$, must be the sum of terms with $p_0$~integrals over~$t$ and $p_1$~integrals over~$V_t$ such that $p_0/\kappa_0+p_1/\kappa_1> 1$\,. 
By bounding the \(p_0+p_1\)~integrals in any such term, the term in the remainder is~\Ord{h^{p_0}\varpi^{-p_1}}.  
Scaling \(h=c_1\varepsilon^{1/\kappa_0}\) and \(\varpi^{-1}=c_2\varepsilon^{1/\kappa_1}\) as \(\varepsilon\to0\) we find terms \(\Ord{h^{p_0}\varpi^{-p_1}}=\ord{h^{\kappa_0}+\varpi^{-\kappa_1}}\) given $p_0/\kappa_0+p_1/\kappa_1> 1$\,. 
Consequently, the remainder \(R_{\kappa_0,\kappa_1}=\ord{h^{\kappa_0}+\varpi^{-\kappa_1}}\).
\end{proof}

As an example of Proposition~\ref{lem:3}, Lemma~\ref{lem:1} proves that~$R_{1,1}$ is~$\Ord{h^2+\varpi^{-2}}=\ord{h^1+\varpi^{-1}}$.
Similarly, from Lemma~\ref{lem:2}, $R_{2,2}$ is~$\Ord{h^3+\varpi^{-3}}=\ord{h^2+\varpi^{-2}}$, consistent with Proposition~\ref{lem:3}. 
Proposition~\ref{lem:3} is flexible because the exponents \(\kappa_0\) and~\(\kappa_1\) need not be identical, nor need be integer. 

An integral expansion for~$u_{t_{n+1}}$ is a useful estimate for~$u_{t_{n+1}}$ provided the integral remainder terms are usefully small: typically this will be for the regime $h,\varpi^{-1}\ll 1$\,.
Therefore, although we emphasise the case where $h\varpi\gg 1$, since this reflects the rapidly oscillating problem described by \ode{}~\eqref{eq:ode}, the approach is not constrained to this regime.
For some exponent $\rho>1$\,, the regime $\varpi^{-1}\sim h^{\rho}$ implies $h\varpi>1$, with more rapid oscillators associated with larger exponents~$\rho$.
However, the case of exponent $0<\rho<1$, resulting in $h\varpi<1$, is also valid and is analogous to \sde{}s, as discussed in Section~\ref{sec:wiener}.

The integral expansion for~$u_{t_{n+1}}$ to errors~$\ord{h^{\kappa_0}+\varpi^{-\kappa_1}}$, in terms of solvable integrals, is compactly written as
\begin{align}
u_{t_{n+1}}&=u_{t_n}+\sum_{(q_0+1)/\kappa_0+q_1/\kappa_1\leq1}
\int_{t_n}^{t_{n+1}}\z \left(\int_{\mathcal{D}} d\vec{\tau}_{q_0+q_1}\right)dt\, a_{t_n}
\nonumber\\&\quad{}
+\sum_{q_0/\kappa_0+(q_1+1)/\kappa_1\leq1}
\int_{V_{t_n}}^{V_{t_{n+1}}}\z \left(\int_{\mathcal{D}} d\vec{\tau}_{q_0+q_1}\right)dV_t\, b_{t_n}+R_{\kappa_0,\kappa_1},
\label{eq:general}
\end{align}
where $d\vec{\tau}_{q_0+q_1}:=\{d\tau_1\ldots d\tau_{q_0+q_1}\}$ with
\begin{equation}
d\tau_j=
\begin{cases}
dt\,\opL_{t_n}^0, \quad 1\leq j\leq q_0, \\
dV_{t}\,\opL_{t_n}^1,\quad q_0+1\leq j\leq q_0+q_1.
\end{cases}
\end{equation}
Here, $\{d\tau_1\ldots d\tau_{q_0+q_1}\}$ represents all unique permutations of the~$d\tau_j$. 
For example, when $q_0=1$ and $q_1=2$ there are three unique permutations, 
\begin{align}
d\vec{\tau}_{1+2}&=\{dt\,\opL_{t_n}^0\,dV_{t}\,\opL_{t_n}^1\,dV_{t}\,\opL_{t_n}^1\}\\
&=dt\,\opL_{t_n}^0\,dV_{t}\,\opL_{t_n}^1\,dV_{t}\,\opL_{t_n}^1+dV_{t}\,\opL_{t_n}^1\,dt\,\opL_{t_n}^0\,dV_{t}\,\opL_{t_n}^1+dV_{t}\,\opL_{t_n}^1\,dV_{t}\,\opL_{t_n}^1\,dt\,\opL_{t_n}^0\nonumber\\
&=dt\,dV_t\,dV_t\,\opL_{t_n}^0\opL_{t_n}^1\opL_{t_n}^1+dV_t\,dt\,dV_t\,\opL_{t_n}^1\opL_{t_n}^0\opL_{t_n}^1
+dV_t\,dV_t\,dt\,\opL_{t_n}^1\opL_{t_n}^1\opL_{t_n}^0.\nonumber
\end{align}
The domain~$\mathcal{D}$ of each integration requires the domain of the integrals over $t$~and~$V_{t}$ to be $[t_n,t]$~and~$[V_{t_n},V_t]$, respectively.
The integrals appearing explicitly in equation~\eqref{eq:general} are straightforwardly evaluated once the microscale oscillation~$v_t$ is specified. 
Figure~\ref{fig:tree} shows a tree diagram representation of equation~\eqref{eq:general} and illustrates some possible choices of exponents~$\kappa$ and~$\rho$ depending upon desired order of accuracy and the relative magnitude of the time step~$h$ and the microscale oscillation time~$\varpi^{-1}$.
Equation~\eqref{eq:general} is essentially the Taylor series expansion of~$u_{t_{n+1}}$ about~$u_{t_n}$ in powers of~$h$ and~$\varpi^{-1}$ so we expect it to be usefully accurate when $h,\varpi^{-1}\ll 1$.

\begin{figure}
\centering
\includegraphics[scale=0.8]{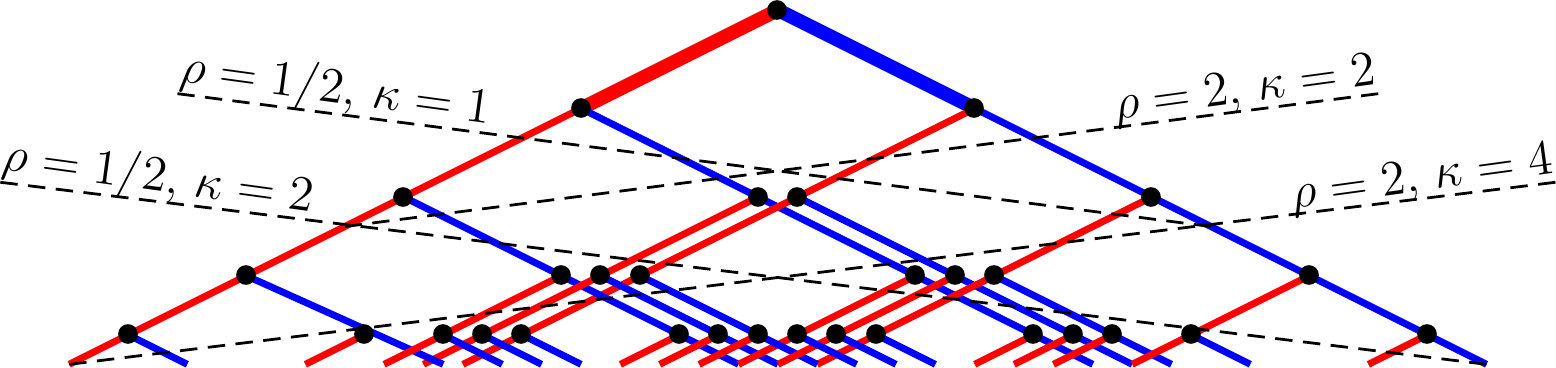}
\caption{A tree diagram of the iterative integration scheme of equation~\eqref{eq:general} where each branch indicates an additional iteration.
Each node represents one term in equation~\eqref{eq:general}, with the root of the tree representing the leading term~$u_{t_n}$.
Each red branch indicates an integral over time~$t$ and each blue branch indicate an integral over~$V_t$:
the thick red branch is $\int_{t_n}^t dt\, a_{t_n}$\,; the thick blue branch is $\int_{V_{t_n}}^{V_t}dV_t\, b_{t_n}$\,; the thin red branches are $\int_{t_n}^t\opL_t^0dt$\,; and the thin blue branches are $\int_{V_{t_n}}^{V_t}\opL_t^1dV_t$\,.
The dashed lines indicate truncations of the expansion with errors~$R_{\kappa,\kappa/\rho}=\Ord{h^\kappa}$ for regimes  $\varpi^{-1}\sim h^{\rho}$ for four examples of the exponents~$\rho$ and order~$\kappa$; the nodes above the dashed line represent the retained terms for each case.}
\label{fig:tree}
\end{figure}

The error of the estimate~\eqref{eq:general} is the local error rather than global error as it is only for one time step of size~$h$ from~$t_n$ to~$t_{n+1}$.
The order of global error, that is, the error over many time steps, is a factor of~$h$ less than the local error, provided certain continuity rules are satisfied.
Specifically, $a(t,u)$, $b(t,u)$ and their derivatives which appear in equation~\eqref{eq:general} must be continuous in~$t$ and Lipschitz continuous in~$u$.
In addition, $v(t)$~must be continuous in~$t$. 
By assuming the \ode~\eqref{eq:ode} has a unique solution we have,  by the Picard--Lindel\"of theorem, already assumed these continuity rules.

For all but the simplest oscillating functions~$v_t$, a potentially computationally expensive part of equation~\eqref{eq:general} are the integrals over~$t$ and~$V_t$, not necessarily the operation of~$\opL^{0,1}$ on~$a_t$ and~$b_t$ (which are simply derivatives).  
One of the main advantages of this iterative integration scheme is that, for a given oscillating function~$v_t$,
once all required integrals over~$V_t$ and~$t$ are evaluated, equation~\eqref{eq:general} is readily computed for any in the family of \ode{}~\eqref{eq:ode} which have different~$a_t$ and~$b_t$, but the same~$v_t$.
Table~\ref{tab:ints} shows the required integrals for the common oscillating functions $v_{1t}=e^{i\varpi t+\phi}$, $v_{2t}=\cos(\varpi t+\phi)$ and $v_{3t}=\sin(\varpi t+\phi)$, with some arbitrary phase~$\phi$. 

\begin{table}
\caption{Indefinite integrals appearing in the time steps \eqref{eq:first}, \eqref{eq:2nd} and~\eqref{eq:general} for three common oscillating functions $v_{1t}=e^{i\varpi t+\phi}$, $v_{2t}=\cos(\varpi t+\phi)$ and $v_{3t}=\sin(\varpi t+\phi)$.
Here, $p$ and~$m$ are non-negative integers, 
but~$K^p_0$ and~$I^p_0$ are always replaced with~$J^p$ when they arise in the integral reduction. 
Negative~$p$ or~$m$ may appear in the integral reduction, in which case, $I^p_m,K^p_m,L^p_m=0$ for negative~$p$ or~$m$.
For oscillating function~$v_{1t}$ the integrals~$I^p_m$ and~$J^p$ are required, whereas for oscillating functions $v_t=v_{2t}$ or $v_t=v_{3t}$ the integrals~$J^p$, $K^p_m$ and~$L^p_m$ are required along with the identity $v_{3t}^{2m}=(1-v_{2t}^2)^m$.
}
\label{tab:ints}
\begin{equation*}
\renewcommand\arraystretch{1.2}
\begin{array}{ll}
\text{integral} & \text{integral reduction}\\
\hline
I^p_{m}=\int t^pv_{1t}^m\,dt& -i\tfrac{1}{m}\varpi^{-1}t^pv_{1t}^m+i\varpi^{-1}\tfrac{p}{m}I^{p-1}_{m}\\
J^p=\int t^p\,dt & \tfrac{1}{p+1}t^{p+1}\\
K^p_{m}=\int t^p v_{2t}^m\,dt & \varpi^{-1}\tfrac{1}{m}t^pv_{2t}^{m-1}v_{3t}-\varpi^{-1}\tfrac{p}{m}L^{p-1}_{m-1}+\tfrac{m-1}{m}K^p_{m-2}\\
L^p_{m}=\int t^pv_{2t}^mv_{3t}\,dt & -\varpi^{-1}\tfrac{1}{m+1}t^pv_{2t}^{m+1}+\varpi^{-1}\tfrac{p}{m+1}K^{p-1}_{m+1}
\end{array}
\end{equation*}
\end{table}

As a low order example of equation~\eqref{eq:general}, consider a case where the oscillation~$v_t$ varies rapidly over the interval~$(t_n,t_{n+1})$ such that $\varpi^{-1}\ll h\ll1$\,.
For illustrative purposes, we choose the regime \(\varpi^{-1}\sim h^2\) and truncate to errors $R_{2,1}=\Ord{h^3}$ (exponents $\rho=2$ and $\kappa=2$ in Figure~\ref{fig:tree}). 
Then the integral expansion~\eqref{eq:general} reduces to 
\begin{align}
u_{t_{n+1}}&=u_{t_n}+a_{t_n}\int_{t_n}^{t_{n+1}}\z dt+\opL_{t_n}^0a_{t_n}\int_{t_n}^{t_{n+1}}\z \int_{t_n}^tdt\,dt+b_{t_n}\int_{V_{t_n}}^{V_{t_{n+1}}}\z dV_t+\Ord{h^3}\nonumber\\
&= u_{t_n}+a_{t_n}h+b_{t_n}(V_{t_{n+1}}-V_{t_n})+\opL_{t_n}^0a_{t_n}h^2/2+\Ord{h^3},\label{eq:p2n2}
\end{align}
where here, and in all following expansions, we replace the remainder term $R_{\kappa_0,\kappa_1}$ with the order of error obtained from Proposition~\ref{lem:3}.
In Figure~\ref{fig:tree} the four nodes appearing above the dashed line labelled ``$\rho=2$, $\kappa=2$'' represent the four terms of equation~\eqref{eq:p2n2}.
In general, for constant order~$\kappa$, continuously varying exponent~$\rho$ in the regime \(\varpi^{-1}\sim h^\rho\) is equivalent to continuously varying the strength of the oscillation relative to the time step~$h$.
Thus, while maintaining a constant order~$\kappa$ so that the order with respect to~$h$ does not change, varying exponent~$\rho$ in equation~\eqref{eq:general} encompasses oscillators of any relative frequency.

\section{Examples}
\label{sec:examples}
We present three examples which demonstrate how one obtains a numerical integration scheme, parametrised by time step~$h$ and oscillation frequency~$\varpi$, from the integrals of equation~\eqref{eq:general}.

\subsection{Purely oscillatory system}
We begin with a scalar \ode\ which only involves a rapidly oscillating term~$v_t$ and the function $b(t,u)=u^{1-\gamma}$, namely
\begin{equation}
\frac{du}{dt}=u^{1-\gamma}v_t\,,\quad u(0)=u_0\,.
\label{eq:powode}
\end{equation}
This \ode\ is readily solved analytically by separation of variables: for any exponent~$\gamma$
\begin{equation}
u_{h}=\begin{cases}
u_{0}\left[1+\gamma u_0^{-\gamma}(V_h-V_{0})\right]^{1/\gamma}, & \gamma\neq 0,\\
u_{0}\exp(V_h-V_{0}), & \gamma=0.\label{eq:solexact}
\end{cases}
\end{equation}

This example illustrates a straightforward implementation of equation~\eqref{eq:general}. 
Without loss of generality, we consider the one time step~$[0,h]$.
The given \ode~\eqref{eq:powode} has $a(t,u)=0$ and $(\opL_{0}^0)^mb_{0}=0$ for positive integer~$m$, so equation~\eqref{eq:general} reduces to 
\begin{align}
u_{h}&=u_{0}+\int_{V_0}^{V_{h}}\left(1+\int_{V_{0}}^{V_t}dV_t\,\opL_{0}^1+\int_{V_{0}}^{V_t}\int_{V_{0}}^{V_t}dV_t\,dV_t\,(\opL_{0}^1)^2
\right. \nonumber\\&\quad\left.{}
+\int_{V_{0}}^{V_t}\int_{V_{0}}^{V_t}\int_{V_{0}}^{V_t}dV_t\,dV_t\,dV_t\,(\opL_{0}^1)^3+\cdots\right)dV_t\,b_{0},
\label{eq:expand}
\end{align}
where we choose the order of the integral expansion to be $q\rightarrow\infty$\,.
Since
\begin{equation}
\int_{V_0}^{V_t}(V_t-V_0)^mdV_t=\frac{1}{m+1}(V_t-V_{0})^{m+1},
\end{equation}
the expansion of~$u_{h}$ in equation~\eqref{eq:expand} simplifies to
\begin{equation}
u_{h}=u_{0}+\sum_{m=1}^{\infty}\frac{1}{m!}(V_{h}-V_{0})^m\left[\left(u^{1-\gamma}\frac{\partial}{\partial u}\right)^{m-1}u^{1-\gamma}\right]_{0}.\label{eq:sol}
\end{equation}
In general, for $m>1$\,,
\begin{align}
\left[\left(u^{1-\gamma}\frac{\partial}{\partial u}\right)^{m-1}u^{1-\gamma}\right]_{0}&=u_{0}^{1-m\gamma}\prod_{s=1}^{m-1}(1-s\gamma)\\
&=\begin{cases}
u_{0}\left(\gamma u_{0}^{-\gamma}\right)^m\frac{\Gamma(1/\gamma+1)}{\Gamma(1/\gamma-m+1)}, & \gamma\neq 0,\\
u_{0}, & \gamma=0,
\end{cases}\nonumber
\end{align}
which, when substituted into equation~\eqref{eq:sol}, produces the Taylor series in time step~$h$ about $t=0$ of the analytic solution~\eqref{eq:solexact} in powers of $(V_{h}-V_0)\sim\varpi^{-1}$.
In this special case the Taylor series is just in powers of~$\varpi^{-1}$ rather than powers of both $h$~and~$\varpi^{-1}$. 

This example is particularly simple in that one evaluates all integrals in equation~\eqref{eq:expand} exactly, without having to truncate the infinite sum of integrals at some particular point.
For more complicated coefficient functions $a(t,u)$~and~$b(t,u)$ one must almost always truncate the series of integrals at some order.

\subsection{Oscillatory system with exponential macroscale}
\label{sec:exp-macro}

This section compares our integral method for solving the \ode{}~\eqref{eq:ode} with a method developed by Condon, Dea\~no and Iserles (\textsc{cdi})~\cite{Condon2009a, Condon2010a, Condon2010b}.
Let's consider nonlinear \ode{}s of the form 
\begin{equation}
\frac{du}{dt}=\alpha(t)u+\mu u^{1-\gamma}v(t),\quad u(0)=u_0\,,\label{eq:expDE}
\end{equation}
which have the exact solution
\begin{align}
u(t)&=\exp\left[\int_0^t \alpha(s)ds\right]\nonumber\\
&\quad{}\times\begin{cases}
\left\{\gamma\mu\int_0^t v(s)\exp \left[-\gamma\int_0^s \alpha(r)dr\right]ds+u_0^{\gamma}\right\}^{1/\gamma},
\, &\gamma\neq 0\,,\\
u_0\exp\left[\mu\int_0^t v(s)ds\right], & \gamma=0\,. 
\end{cases}
\end{align}
If the integrals in the above solutions cannot be solved analytically, they may be solved numerically using a Filon quadrature~\cite{Iserles2006}.
The macroscale behaviour of~$u(t)$ would be exponential for real~$\alpha(t)$ and sinusoidal when $\alpha(t)$~is imaginary.
The rapid microscale oscillations are superimposed on the macroscale.
Figure~\ref{fig:egOsc} plots two examples of solutions to the \ode~\eqref{eq:expDE}. 

The \cdi~method~\cite{Condon2009a, Condon2010a, Condon2010b} 
expands~$u(t)$ in terms of powers of~$\varpi^{-1}$, with the coefficients of these powers written in terms of a Fourier expansion,
\begin{equation}
u(t)=\sum_{r=0}^{\infty}\frac{1}{\varpi^{r}}\psi_r(t)
\quad\text{where}\quad
\psi_r(t)=\sum_{j=-\infty}^{\infty}\psi_{r,j}(t)e^{ij\varpi t}.\label{eq:ut}
\end{equation}
The oscillating function is also written as a Fourier expansion,
\begin{equation}
v_t=\sum_{j=-\infty}^{\infty}a_j(t)e^{ij\varpi t}.\label{eq:vt}
\end{equation}
On substituting equations~\eqref{eq:ut} and~\eqref{eq:vt} into the \ode~and equating similar powers of $\varpi^{-1}$~and~$e^{i\varpi t}$, one obtains equations to solve for the coefficients~$\psi_{r,j}(t)$:
each $\psi_{r,0}(t)$~is obtained from a first order \ode;
the remaining coefficients, $\psi_{r,j}(t)$ for $j\neq 0$, are a function of the coefficients $\psi_{q,i}(t)$ for $q<r$ for all~$i$.
For a solution of~$u(t)$ correct to~$\Ord{\varpi^{-m}}$, a possible disadvantage is that one must solve $(m+1)$~\ode{}s. 
Although the \cdi{} method appears quite cumbersome in its general form, 
the number of \ode{}s to be solved increases linearly with order, as discussed in more detail in Section~\ref{sec:nc}.
Furthermore,  the \ode{}s for the~$\psi_{r,0}(t)$ do not contain any rapidly oscillating terms so are readily solved by standard numerical methods. 

A disadvantage of the \cdi\ method is that significant pre-processing has to be done for every new differential equation to which the method is applied.
In contrast, apart from evaluating derivatives of the coefficient functions $a$~and~$b$, our method only needs new pre-processing if one changes the rapidly oscillating function~$v(t)$.

\subsubsection{Linear case}
\label{sec:lc}


We set $\alpha(t)=t$ and exponent $\gamma=1$\,,  so that $a(t,u)=ut$ and $b(t,u)=\mu$ is constant,  with sinusoidal rapid oscillations $v(t)=\cos\varpi t$\,.
The resulting linear \ode~is
\begin{equation}
\frac{du}{dt}=ut+\mu\cos\varpi t\,, \quad u(0)=u_{0}\,.
\label{eq:simple}
\end{equation}
Using both our integral method and the method of \textsc{cdi}, we evaluate this \ode~over the interval~$[0,h]$ correct to errors $R_{4,2}=\Ord{h^5}$ assuming $\varpi^{-1}\sim h^2$; that is, exponent $\rho=2$ and order $\kappa=4$ in Figure~\ref{fig:tree}.  

In the \cdi~method the nonzero~$a_j$ are $a_1=a_{-1}=\mu/2$.
Up to second order in~$\varpi^{-1}$ the \ode{}s are 
$d\psi_{r,0}(t)/dt=t\psi_{r,0}(t)$, for $r=0,1,2$, with initial conditions $\psi_{0,0}(0)=u_0$ and $\psi_{1,0}(0)=\psi_{2,0}(0)=0$. 
For this case there are only four other nonzero coefficients, $\psi_{1,\pm1}$ and $\psi_{2,\pm1}$.
Thus the \cdi\ estimate for \ode~\eqref{eq:simple} at $t_1=h$ is
\begin{equation}
u_{h}=u_{0}\exp(h^2/2)+\varpi^{-1}\mu\sin\varpi h-\varpi^{-2}h\mu\cos\varpi h+\Ord{\varpi^{-3}}.\label{eq:de}
\end{equation}

Solving the particular \ode~\eqref{eq:simple} using our integral method requires equation~\eqref{eq:general} with exponent $\rho=2$ and order $\kappa=4$\,,
\begin{align}&
u_{h}=u_{0}+\int_{0}^{h}\left[1+\int_{V_{0}}^{V_t}dV_t\, \opL_{0}^1+\int_{0}^tdt\, \opL_{0}^0+\int_{V_{0}}^{V_t}\int_{0}^tdt\,dV_t\, \opL_{0}^0\opL_{0}^1\right.
\nonumber\\&\left.{}
+\int_{0}^t\int_{V_{0}}^{V_t}dV_t\,dt\, \opL_{0}^1\opL_{0}^0+\int_{0}^t\int_{0}^tdt\,dt\, (\opL_{0}^0)^2+\int_{0}^t\int_{0}^t\int_{0}^tdt\,dt\,dt\, (\opL_{0}^0)^3\right]dt\,a_{0}  
\nonumber\\&{}
+\int_{V_{0}}^{V_{h}}\left[1+\int_{0}^tdt\, \opL_{0}^0+\int_{0}^t\int_{0}^tdt\,dt\, (\opL_{0}^0)^2+\int_{V_{0}}^{V_t}dV_t\, \opL_{0}^1\right]dV_t\,b_{0} +\Ord{h^5}.
\label{eq:n4p2}
\end{align}
We firstly calculate the relevant operations of~$\opL^{j}_t$ on the coefficient functions $a(t,u)=ut$ and $b(t,u)=\mu$:
\begin{align}
\opL^0_ta_t&=(1+t^2)u_t,  &(\opL^0_t)^2a_t&=(3+t^2)tu_t,\nonumber\\
(\opL^0_t)^3a_t&=(3+6t^2+t^4)u_t,  &\opL_t^1a_t&=\mu t,\nonumber\\
\opL^0_t\opL^1_ta_t&=\mu,  &\opL^1_t\opL^0_ta_t&=\mu(1+t^2),\nonumber\\
\opL^{0,1}_tb_t&=0,  &(\opL^0_t)^2b_t&=0.\label{eq:Ls}
\end{align}
On substituting these with $t=0$ into equation~\eqref{eq:n4p2}, evaluating all integrals, using Table~\ref{tab:ints},  we obtain
\begin{equation}
u_{h}=u_{0}(1+h^2/2+h^4/8)+\varpi^{-1}\mu\sin\varpi h+\Ord{h^5}.\label{eq:int}
\end{equation}
Further work shows that the $\Ord{h^5}$ and~$\Ord{h^6}$ corrections are 
$-\varpi^{-2} h\mu\cos\varpi h$ and $u_0h^6/48+\varpi^{-3}\mu\sin\varpi h$, 
respectively.

The two estimates in equations~\eqref{eq:de} and~\eqref{eq:int} obtained via the two different methods are not identical.
One reason for the difference is that the \cdi~method only involves an expansion in the microscale time~$\varpi^{-1}$, whereas our integral method involves an expansion in both $\varpi^{-1}$~and~$h$.
Consequently, the first term in equation~\eqref{eq:de} is~$u_{0}\exp(h^2/2)$, but in equation~\eqref{eq:int} this term is replaced by a Taylor expansion in~$h$ with error~$\Ord{h^5}$.
The different expansions also affect how the solutions are truncated: the \textsc{cdi} estimate has error~$\Ord{\varpi^{-3}}$ but no apparent $h$~dependent error; whereas the other estimate has error $R_{4,2}=\Ord{h^5}=\Ord{\varpi^{-5/2}}$.
Thus, the term $\varpi^{-2} h\mu\cos\varpi h$ appears in the \textsc{cdi} estimate but not in the integral method estimate because in the former it is less than the required order, but in the latter it is not.

\subsubsection{A nonlinear case}
\label{sec:nonlin}
Let's choose the case of \ode~\eqref{eq:expDE} with exponent $\gamma=-1$ and constant $\alpha(t)=\alpha$ so that $a(t,u)=\alpha u$ and $b(t,u)=\mu u^2$.
We also choose complex rapid oscillations $v(t)=e^{i\varpi t}$ so that the \ode~\eqref{eq:expDE} becomes
\begin{equation}
\frac{du}{dt}=\alpha u+\mu u^2e^{i\varpi t}, \quad u(0)=u_{0}\,.
\label{eq:nonlinode}
\end{equation}
As in the linear case, section~\ref{sec:lc}, we solve this nonlinear \ode\ over the time interval~$[0,h]$ correct to errors $R_{4,2}=\Ord{h^5}$, assuming $\varpi^{-1}\sim h^2$.  

For the \cdi~method the only nonzero $a_j$~coefficient is $a_1=1$.
The \ode{}s for~$\psi_{r,0}$ are trivial, as in the linear example of section~\ref{sec:lc};\footnote{The \cdi{} method involves a Fourier expansion of~$v(t)$ and derivatives of the functions~$a(t,u)$ and~$b(t,u)$ with respect to~$u$.
Therefore, the \cdi\ method is particularly simple when $v(t)$~is exponential or sinusoidal, and $a(t,u)$ and~$b(t,u)$ are small powers of~$u$.} 
namely, $d\psi_{r,0}(t)/dt=\alpha \psi_{r,0}(t)$ for $r=0,1,2$ and $\psi_{0,0}=u_0$, $\psi_{1,0}=-\psi_{1,1}(0)$, $\psi_{2,0}=-[\psi_{2,1}(0)+\psi_{2,2}(0)]$.
On evaluating all $\psi_{r,j}$~coefficients, at time $t_1=h$
\begin{align}
u_h&=u_0e^{\alpha h}+\varpi^{-1}(1-v_he^{\alpha h})i\mu u_0^2e^{\alpha h}\nonumber\\
&\quad{}
+\varpi^{-2}[-(\alpha +\mu u_0)+(\alpha +2\mu u_0)v_he^{\alpha h}-\mu u_0v_h^2e^{2\alpha h}]\mu u_0^2e^{\alpha h}
\nonumber\\&\quad{}
+\Ord{\varpi^{-3}}.\label{eq:nonlin}
\end{align}

For the iterative integral method, the differential operations of~$\opL_t^{j}$ on~$a(t,u)$ and~$b(t,u)$ at the initial time $t=0$ are
\begin{align}
(\opL_0^0)^n(\opL_0^1)^ma_0&=\alpha^{n+1}\mu^m m!u_0^{m+1},\nonumber\\
(\opL_0^0)^n(\opL_0^1)^mb_0&=2^n\alpha^n\mu^{m+1}(m+1)!u_0^{m+2}.\label{eq:ops}
\end{align}
Table~\ref{tab:ints} provides the required integrals, $I^p_m$ and~$J^p$ with phase $\phi=0$.
Substituting the integrals and equation~\eqref{eq:ops} into equation~\eqref{eq:general} produces 
\begin{align}
u_h&=u_0+h(1+h\alpha/2+h^2\alpha^2/6+h^3\alpha^3/24)\alpha u_0\nonumber\\&\quad{}
+\varpi^{-1}[(1+\alpha h+\alpha^2 h^2/2)-v_h(1+2\alpha h+2\alpha^2 h^2)]i\mu u_0^2
\nonumber\\&\quad{}
-\varpi^{-2}[(1-v_h)\mu u_0+\alpha](1-v_h)\mu u_0^2+\Ord{h^5},\label{eq:nonlinsol}
\end{align}
Again this expression is the Taylor series expansion of the \cdi~\eqref{eq:nonlin} in~$h$, as expected.
The $\Ord{h^5}$~correction is
\begin{eqnarray}&&
h^5\alpha^5u_0/120+h^3\varpi^{-1}(1-8v_h)i\alpha^3\mu u_0^2/6
\nonumber\\&&{}
-h\varpi^{-2}[\alpha(1-2v_h)+(1-v_h)(1-3v_h)\mu u_0]\alpha\mu u_0^2.
\end{eqnarray}

\subsection{Cater for unknown microscale phase}

The microscale oscillations may be so fast that we do not know the phase of the oscillations: in modelling oscillations we know that phases easily drift but amplitudes are much more robust \cite[e.g.]{Abrams06, Brown2004, Cross93}.
Further, a small uncertainty in the frequency will, over the many oscillations in one time step~$h$,  manifest itself as a de-correlation of the phase of~$v(t)$ at the end of the time step~$h$ compared to that at the beginning.
An average over all phases reflects a modelling of such de-correlation.
Thus this section addresses issues arising from uncertain phases of the microscale oscillations.

Suppose the oscillation~$v_t$ includes an unknown `random' phase~$\phi$ which we accommodate in analysis by replacing~$v(t)$ by~$v(t+\phi)$.
In this case, the procedure for finding the series expansion of the \ode{} solution does not change.
Once $u_{t_{n+1}}$~has been obtained as a function of~$\phi$, an average over all~$\phi$ is performed, defined by
\begin{equation}
\langle\cdot\rangle_{\phi}=\left.\int_{\phi}(\cdot)d\phi\right/\int_{\phi}d\phi,
\end{equation}
where the subscript~$\phi$ on the integrals refers to the domain of the phase~$\phi$.

For example, consider the \ode{}~\eqref{eq:ode} with general functions~$a(t,u)$ and~$b(t,u)$ and $v(t)=\cos(\varpi t+\phi)$.
Both  $a(t,u)$ and~$b(t,u)$ are independent of~$\phi$.
The $\phi$-averaged solution at~$t_{n+1}$, for exponent $\rho=2$ and order $\kappa=4$ in Figure~\ref{fig:tree} and corresponding to error $R_{4,2}$, is obtained by averaging equation~\eqref{eq:general} over all phases~$\phi$,
\begin{align}
\langle u_{t_{n+1}}\rangle_{\phi}&=\langle u_{t_n}\rangle_{\phi}+\left\langle\int_{t_n}^{t_{n+1}}\left(1+\int_{t_n}^tdt\, \opL_{t_n}^0+
+\int_{t_n}^t\int_{t_n}^tdt\,dt\, (\opL_{t_n}^0)^2\right.\right.\nonumber\\
&\quad{}\left.\left.+\int_{t_n}^t\int_{t_n}^t\int_{t_n}^tdt\,dt\,dt\, (\opL_{t_n}^0)^3\right)dt\,a_{t_n} \right\rangle_{\phi} 
\nonumber\\
&\quad{}+\left\langle\int_{V_{t_n}}^{V_{t_{n+1}}}\left(\int_{V_{t_n}}^{V_t}dV_t\, \opL_{t_n}^1\right)dV_t\,b_{t_n}\right\rangle_{\phi} +\Ord{h^5}\nonumber\\
&=\langle u_{t_n}\rangle_{\phi}+a_{t_n}h+\opL^0_{t_n}a_{t_n}h^2/2+(\opL^0_{t_n})^2a_{t_n}h^3/6+(\opL^0_{t_n})^3a_{t_n}h^4/24\nonumber\\
&\quad{}+\opL^1_{t_n}b_{t_n}[1-\cos(\varpi h)]\varpi^{-2}/2+\Ord{h^5}.
\label{eq:phased}
\end{align}
In the above, we neglect all single integrals over~$V_t$ since they vanish after averaging over~$\phi$.
If higher order accuracy is required, Table~\ref{tab:ints} provides the relevant integrals, $J^p$, $K^p_m$ and~$L^p_m$.
Our integral expansion approach empowers the resolution of macroscale effects generated by microscale interactions, the last line of equation~\eqref{eq:phased}, without resolving all the complexity of microscale details, and in the presence of microscale uncertainty.

\subsection{Frequency dependent coefficients}
\label{sec:freqdep}
The oscillating function~$v(t)$ may have an amplitude which varies with the frequency, say $v(t)=\Ord{\varpi^{-\nu}}$.
This may describe situation where certain frequencies are attenuated by a filter.
For example, in electrical circuits a filter may affect all frequencies within a given range and cause the amplitude of the voltage across some circuit element to decrease in some frequency dependent way.
Possible examples include $v(t)=\varpi^{-1}\cos\varpi t$, where $\nu=1$ and the amplitude decreases with frequency, or $v(t)=\varpi^{1/2}e^{i\varpi t}$, where $\nu=-1/2$ and the amplitude increases with frequency.  
This case is roughly analogous with the noise term in \sde{}s, where on a microscale time scale~$dt$ the  stochastic fluctuations of the noise have `amplitude'${}\propto dt^{-1/2}$ (so that increments are${}\propto\sqrt{dt}$): 
here the microscale $dt\sim\varpi^{-1}$ so the analogous amplitude scales like~$\sqrt\varpi$; that is, the exponent $\nu=-1/2$. 
In essence we make predictions at finite large frequency~$\varpi$ through integral expansions truncated to reflect different distinguished limits, limits where the oscillations also become large.

For $v(t)=\Ord{\varpi^{-\nu}}$ one proceeds as before but must reconsider the order of each~$v(t)$ dependent term in the integral expansion~\eqref{eq:general}.
Recall that each integral over~$V_t$ is originally~$\Ord{\varpi^{-1}}$.
Now, with the additional factor of~$\varpi^{-\nu}$, each integral over~$V_t$ is~$\Ord{\varpi^{-(1+\nu)}}$.
Therefore, a term with $q_1$~integrals over~$V_t$ was previously~$\Ord{\varpi^{-q_1}}$ but is now~$\Ord{\varpi^{-q_1(1+\nu)}}$.
To reasonably ensure the higher order terms that appear in the corresponding residual (that is, those involving many integrals over $V_t$~and~$t$) are negligible compared to the lower order terms,
$\varpi^{-q_1(1+\nu)}$~should decrease with increasing~$q_1$.
As $\varpi^{-1}<1$ and $q_1>0$, we thus require $\nu>-1$.
To generalise equation~\eqref{eq:general} for $v(t)=\Ord{\varpi^{-\nu}}$ one 
replaces~$\kappa_1$ with~$\kappa'_1$ and defines $\kappa'_1=\kappa_1/(\nu+1)$.
The error of this generalised version of equation~\eqref{eq:general} is $R_{\kappa_0,\kappa'_1}=\ord{h^{\kappa_0}+\varpi^{-\kappa'_1}}$, from Proposition~\ref{lem:3}.

For example, consider the family of \ode{}s
\begin{equation}
\frac{du}{dt}=\alpha u+\varpi^{1/2}\mu u^2e^{i\varpi t}, \quad u(0)=u_{0}\,.
\label{eq:nonlinear}
\end{equation}
Each such \ode{} is identical to the nonlinear example in  Section~\ref{sec:nonlin}, with the exception that here we choose $v(t)=\varpi^{1/2}e^{i\varpi t}$ to have a frequency dependent amplitude.
We again solve over the interval~$[0,h]$ correct to errors~$\Ord{h^5}$, assuming $\varpi^{-1}\sim h^2$.
For this case, $\rho=2$, $\kappa=4$ and $\nu=-1/2$ so $\kappa_0=\kappa=4$ and $\kappa_1'=\kappa_1/(\nu+1)=\kappa/\rho(\nu+1)=4$ and the error is $R_{4,4}=\Ord{h^5}$.
After substituting into equation~\eqref{eq:general} with $\rho$ replaced by~$\rho'$ and evaluating all terms, we obtain the time step rule
\begin{align}
u_h=& S_4(h') u_0+\varpi^{-1/2}[S_3(h')
-v'_hS_3(2h')]i\mu u_0^2\nonumber\\
&{}-\varpi^{-1}[S_2(h')-2v'_hS_2(2h')+v_h'^2S_2(3h')]\mu^2u_0^3\nonumber\\
&{}-\varpi^{-3/2}[S_1(h')-v'_hS_1(2h')]\alpha\mu u_0^2\nonumber\\
&{}-\varpi^{3/2}[S_1(h')-3v'_hS_1(2h')+3v_h'^2S_1(3h')-v_h'^3S_1(4h')]i\mu^3 u_0^4\nonumber\\
&{}-\varpi^{-2}(1-v'_h)^22i\alpha\mu^2u_0^3+\varpi^{-2}(1-v'_h)^4\mu^4u_0^5
+\Ord{h^5}
\end{align}
where $h'=h\alpha$, $v_h'=e^{i\varpi h}$ and the Taylor polynomial $S_n(x)=\sum_{j=0}^{n}x^j/j!$\,.
Thus our approach flexibly adapts to many different parameter regimes.

\section{Numerical considerations}
\label{sec:nc}

The \cdi~method and the integral method are both recursive so are scalable to higher orders when implementing a numerical solution.
However, if the \ode\ changes even slightly then all pre-processing calculations must be redone in the \cdi~method;
further the \cdi~method does not appear to have much scope for parallelisation as higher order terms depend explicitly on lower order terms.

In our integral approach, for a given oscillation~$v_t$, we need to compute integrals of the form
\begin{equation}
\int_{t_n}^{t}(s-t_n)^pv_s^mds,\label{eq:ints}
\end{equation}
for non-negative integers~$p,m$, where the highest values of~$p$ and~$m$ are determined by the desired order of accuracy of the solution~$u_{t_{n+1}}$.
Some examples of these integrals are shown in Table~\ref{tab:ints}.
The above integrals are only calculated \emph{once} for any given oscillation~$v_t$ in the pre-processing.
The numerical simulation for solving \ode{}~\eqref{eq:ode} then simply involves evaluating operations of~$\opL^0_t$ and~$\opL^1_t$ on~$a_t$ and~$b_t$ at $t=t_n$, (which are straightforward derivatives) and substitution into equation~\eqref{eq:general}.
While the evaluation of integrals~\eqref{eq:ints} for a given~$v_t$ may be computationally expensive, possibly requiring extensive numerical calculations (for example, $v(t)=e^{i\cos\varpi t}$), once they are evaluated one can quickly solve for a family of \ode{}s~\eqref{eq:ode} with the same~$v(t)$ but different~$a(t,u)$ and~$b(t,u)$.
This contrasts with analogous numerical schemes for \sde{} where the corresponding stochastic integrals need to be computed on the fly since stochastic effects are independent in every time step and between every realisation.

Of particular importance to a numerical implementation is the increase in the number of terms as the order of the estimate is increased.
For the \cdi~method $\varrho\in\mathbb{N}$ is defined such that $a_j=0$ for all $|j|\geq \varrho+1$, and the maximum number of $\psi_{r,j}$~terms requiring calculation for a given~$r$ is $2r\varrho+1$~\cite{Condon2010b}.
For  example, in Section~\ref{sec:exp-macro}, $\varrho=1$ and so each~$\psi_r$ introduces up to $2r+1$~terms.
Recall that for this method~$\psi_r$ is the coefficient of~$\varpi^{-r}$.
Therefore, increasing the order of the solution from~$\Ord{\varpi^{-n+1}}$ to~$\Ord{\varpi^{-n}}$ requires, in general, a linear increase in number of $\psi_{r,j}$~terms of $2n\varrho+1$.  

For our integral method with error $R_{\kappa_0,\kappa_1}$ where $\kappa_0=\kappa$ and $\kappa_1=\kappa/\rho$, the number of integral terms to be calculated is
\begin{equation}
N(\kappa,\rho)=\sum_{i=0}^{\kappa}\sum_{j=0}^{\rho(\kappa-i)}\frac{(i+j)!}{i! j!}-1.
\end{equation}
For $\rho>1$ an increase in the order of the estimate from $\Ord{\varpi^{-n+1}}=\Ord{h^{(n-1)\rho}}$ to $\Ord{\varpi^{-n}}=\Ord{h^{n\rho}}$ results in an increase in the number of integral terms which is significantly more than the linear increase of the \cdi~method.
In this sense the integral method appears less efficient than the \cdi~method; however, these integrals are done only once as a pre-processing step, and are thereafter useful to solve a large family of \ode{}s.
For example, $N(n,1)=2(2^{n}-1)$ and so increasing the order from $\Ord{\varpi^{-n+1}}=\Ord{h^{(n-1)}}$ to $\Ord{\varpi^{-n}}=\Ord{h^{n}}$ the increases the number of integral terms by~$2^n$.
The larger the value of~$\rho$, the greater the increase in terms so~$2^n$ for $\rho=1$ is a lower bound for the increase in terms when $\rho>1$. 

\section{Relate stochastic Wiener process to oscillations}
\label{sec:wiener}

We have discussed the case of a rapid oscillator with a well defined and very short period of oscillation~$\varpi^{-1}$ such that $\varpi^{-1}\ll h<1$.
To conveniently truncate the expansions in both~$h$ and~$\varpi^{-1}$ we often define an exponent~$\rho$ such that $\varpi^{-1}\sim h^\rho$ and require $\rho>1$.
Larger exponents~$\rho$ are associated with higher frequency oscillators. 
In contrast, a stochastic process such as a Wiener process is noisy and has no well defined oscillation.
A noise term has many relevant, but unspecified, short and long time scales.
When expanding in terms of these time scales, it is the longer time scales (corresponding to slow `frequencies') which determine the order of a given term.
Therefore, for truncation purposes, only the slowest frequencies are relevant and these are defined by~$\varpi$.
An additional complication is that the amplitude of the noise is frequency dependent and typically, for noise with time scale~$\varpi^{-1}$, with amplitude~$\Ord{\varpi^{-\nu}}$ with $\nu=-1/2$, as discussed in Section~\ref{sec:freqdep}.
In general,
$V_{t_{n+1}}-V_{t_{n}}\sim \varpi^{-(\nu+1)}\sim h^{\rho'}$ for some $0<\rho'<1$ where $\rho'=\rho(\nu+1)$\,.

We now show how equation~\eqref{eq:general} connects to two stochastic schemes, the Euler scheme and the Milstein scheme, which are both used to solve Ito stochastic differential equations of the form given in equation~\eqref{eq:ode} but with $V_t$~replaced by a Wiener process~$W_t$ \cite[e.g.]{Kloeden2001, Iacus2008}.
We still require $h,\varpi^{-1}\ll 1$ so that the expansion is valid. We set $\kappa_0=\kappa$ and $\kappa_1'=\kappa_1/(\nu+1)=\kappa/\rho(\nu+1)=\kappa/\rho'$.
The Euler scheme is reproduced from equation~\eqref{eq:general}  when 
$1/2<\rho'<1$ and $1\leq \kappa<2\rho'$\,,
\begin{align}
u_{t_{n+1}}&=u_{t_n}+a_{t_n}\int_{t_n}^{t_{n+1}}\z dt+b_{n}\int_{W_{t_n}}^{W_{t_{n+1}}}\z dW_t+\Ord{h^{\varepsilon}}\nonumber\\
&=u_{t_n}+a_{t_n}(t_{n+1}-t_n)+b_{t_n}(W_{t_{n+1}}-W_{t_n})+R_{\kappa,\kappa/\rho'}\,.
\end{align}
When $\nu=-1/2$\,, $\rho'=\rho/2$ and $1\leq \kappa<\rho<2$\,. 
The Milstein scheme is 
reproduced from equation~\eqref{eq:general} when $1/3<\rho'<1$ and $1,2\rho'\leq \kappa<1+\rho'\,,3\rho'$\,,
\begin{align}
u_{t_{n+1}}&=u_{t_n}+a_{t_n}\int_{t_n}^{t_{n+1}}\z dt+b_{t_n}\int_{W_{t_n}}^{W_{t_{n+1}}}\z dW_t+\opL_{t_n}^1b_{t_n}\int_{W_{t_n}}^{W_{t_{n+1}}}\z \int_{W_{t_n}}^{W_t}dW_t\,dW_t+R_{\kappa,\kappa/\rho'}\nonumber\\
&=u_{t_n}+a_n(t_{n+1}-t_n)+b_{t_n}(W_{t_{n+1}}-W_{t_n})\nonumber\\
&\quad{}+b_{t_n}(\partial b/\partial u)_{t_n}\left[(W_{t_{n+1}}-W_{t_n})^2-(t_{n+1}-t_n)\right]/2+R_{\kappa,\kappa/\rho'}\,,
\end{align}
where the final integral is evaluated using Ito's lemma. When $\nu=-1/2$\,, $2/3<\rho<2$ and $1,\rho\leq\kappa<1+\rho/2\,,3\rho/2$\,.
One can easily improve on these two schemes by choosing larger~$\kappa$ (resulting in a higher order of accuracy) and smaller~$\rho$ (accounting for longer time scales in the noise term) in equation~\eqref{eq:general}.

\section{Conclusion}

We propose a straightforward methodology for integrating \ode{}s which contain rapidly oscillating factors.
These \ode{}s are not to be confused with smooth \ode{}s which have highly oscillatory solutions. 
Our method requires repeated iterations of the integral version of the chain rule, akin to that used for \sde{}s.
The method gives an estimate and a remainder for any time step~$h$ and period of oscillation~$\varpi^{-1}$.
The estimate over a time step is obtained by evaluating a series of straightforward integrals over time~$t$ and the oscillation~$V_t$, in terms of derivatives of the smooth coefficient functions which appear in the original differential equation.
The remainder gives an exact expression for the error to provide a bound in any given application.
Such rapidly oscillating systems require $\varpi^{-1}\ll h<1$\,, but our method is also applicable to any case within the limit $\varpi^{-1},h<1$\,.

We expect the method presented here to adapt to more complex problems such as higher order differential equations and differential equations involving multiple rapid oscillators~\cite{Condon11}.
Another possibility for future research is the development of a derivative free scheme:
the scheme presented here requires the computation of derivatives of~$a(t,u)$ and~$b(t,u)$, which is may be inconvenient in applications.

\paragraph{Acknowledgement}
This research was supported by grant DP120104260 from the Australian Research Council.

\end{document}